\documentclass[a4paper]{amsproc}

\usepackage{amssymb}
\usepackage{pstricks}
\usepackage{pstcol}
\usepackage{graphicx}
\usepackage{amsmath}
\usepackage{amsfonts}
\usepackage{latexsym}

\theoremstyle{plain}

\theoremstyle{definition}

\numberwithin{equation}{section}

\title[Trigonometric beta integrals]{The integrals in Gradshteyn and Ryzhik. \\
Part 13: Trigonometric forms of the beta function}

\subjclass[2000]{Primary 33}

\keywords{Integrals, Beta function}

\author[V. Moll]{Victor H. Moll}
\address{Department of Mathematics,
Tulane University, New Orleans, LA 70118}
\email{vhm@math.tulane.edu}

\address{\hfill{\it Received 7 07
2009, revised ?? }\newline Departamento de Matem\'atica
\newline
Universidad T\'ecnica Federico Santa Mar\'{\i}a
\newline  Casilla 110-V,
\newline Valpara\'{\i}so, Chile}

\thanks{The author wishes to acknowledge the partial support of  
NSF-DMS 0713836.}

\begin{document}

{\begin{flushleft}\baselineskip9pt\scriptsize {\bf SCIENTIA}\newline
Series A: {\it Mathematical Sciences}, Vol. ?? (2010), ??
\newline Universidad T\'ecnica Federico Santa Mar{\'\i}a
\newline Valpara{\'\i}so, Chile
\newline ISSN 0716-8446
\newline {\copyright\space Universidad T\'ecnica Federico Santa
Mar{\'\i}a\space 2010}
\end{flushleft}}

\vspace{10mm} \setcounter{page}{1} \thispagestyle{empty}

\begin{abstract}
The table of Gradshteyn and Rhyzik contains some  
trigonometric integrals  that can be 
expressed in terms of the beta function. We describe the
evaluation of some of them.
\end{abstract}

\maketitle

\newcommand{\nn}{\nonumber}
\newcommand{\ba}{\begin{eqnarray}}
\newcommand{\ea}{\end{eqnarray}}
\newcommand{\ift}{\int_{0}^{\infty}}
\newcommand{\ione}{\int_{0}^{1}}
\newcommand{\ifft}{\int_{- \infty}^{\infty}}
\newcommand{\no}{\noindent}
\newcommand{\realpart}{\mathop{\rm Re}\nolimits}
\newcommand{\imagpart}{\mathop{\rm Im}\nolimits}

\newtheorem{Definition}{\bf Definition}[section]
\newtheorem{Thm}[Definition]{\bf Theorem} 
\newtheorem{Example}[Definition]{\bf Example} 
\newtheorem{Lem}[Definition]{\bf Lemma} 
\newtheorem{Note}[Definition]{\bf Note} 
\newtheorem{Cor}[Definition]{\bf Corollary} 
\newtheorem{Prop}[Definition]{\bf Proposition} 
\newtheorem{Problem}[Definition]{\bf Problem} 
\numberwithin{equation}{section}

\maketitle

\section{Introduction} \label{intro} 
\setcounter{equation}{0}

The table of integrals \cite{gr} contains a large variety of definite 
integrals in trigonometric form that can be evaluated in terms of the 
{\em beta function} defined by
\begin{equation}
B(a,b) = \int_{0}^{1} x^{a-1} (1-x)^{b-1} \, dx. 
\label{beta-def}
\end{equation}
\noindent
The convergence of the integral requires $a, \, b  > 0$. 

The change of variables $x= \sin^{2}t$ yields the basic representation
\begin{equation}
B(a,b) = 2 \int_{0}^{\pi/2} \sin^{2a-1}t \, \cos^{2b-1}t \, dt,
\end{equation}
\noindent
that, after replacing $(2a,2b)$ by $(a,b)$, is written as
\begin{equation}
\int_{0}^{\pi/2} \sin^{a-1}t \, \cos^{b-1}t \, dt = 
\frac{1}{2} B \left( \frac{a}{2}, \frac{b}{2} \right).
\label{36215}
\end{equation}
\noindent
This appears as $\mathbf{3.621.5}$ in \cite{gr}.

\section{Special cases} \label{sec-special} 
\setcounter{equation}{0}

In this section we present several special cases of formula (\ref{36215}) 
that appear in \cite{gr}. 

\begin{Example}
The choice $b=1$ in (\ref{36215}) gives
\begin{equation}
\int_{0}^{\pi/2} \sin^{a-1}t \, dt = \frac{1}{2} B \left( \frac{a}{2}, 
\frac{1}{2} \right).
\label{form-1}
\end{equation}
\noindent
Legendre's duplication formula 
\begin{equation}
\Gamma(2a) = \frac{2^{2a-1}}{\sqrt{\pi}} \Gamma(a) \Gamma( a + \tfrac{1}{2}) 
\label{legen-1}
\end{equation}
\noindent
can be used to write (\ref{form-1}) as
\begin{equation}
\int_{0}^{\pi/2} \sin^{a-1}t \, dt = 2^{a-2} B \left( \frac{a}{2},
\frac{a}{2} \right) = \frac{2^{a-2} \, \Gamma^{2}(a/2)}{\Gamma(a)}.
\label{form-2}
\end{equation}
\noindent
This is $\mathbf{3.621.1}$ in \cite{gr}. The dual evaluation
\begin{equation}
\int_{0}^{\pi/2} \cos^{a-1}t \, dt = 2^{a-2} B \left( \frac{a}{2}, 
\frac{a}{2} \right) = \frac{2^{a-2} \, \Gamma^{2}(a/2)}{\Gamma(a)}, 
\label{form-2cos}
\end{equation}
\noindent
comes from the change of variables $t \mapsto \frac{\pi}{2} - t$. 
The reader will find a proof of (\ref{legen-1}) in \cite{irrbook}. 
\end{Example}

\begin{Example}
The special case $a=\tfrac{1}{2}$ in (\ref{form-2}) gives $\mathbf{3.621.7}$:
\begin{equation}
\int_{0}^{\pi/2} \frac{dx}{\sqrt{\sin x }} = \frac{\Gamma^{2} \left( 
\tfrac{1}{4} \right)}{2 \sqrt{2 \pi}}.
\end{equation}
\end{Example}

\begin{Example}
The special case $a=\tfrac{3}{2}$ in (\ref{form-2}) gives $\mathbf{3.621.6}$:
\begin{equation}
\int_{0}^{\pi/2} \sqrt{\sin x } \, dx  = \sqrt{\frac{2}{\pi}} 
\Gamma^{2} \left( \tfrac{1}{4} \right). 
\end{equation}
\end{Example}

\begin{Example}
The special case $a=\tfrac{5}{2}$ in (\ref{form-2}) gives $\mathbf{3.621.2}$:
\begin{equation}
\int_{0}^{\pi/2} \sin^{3/2} x  \, dx  = \frac{1}{6 \sqrt{2 \pi}} 
\Gamma^{2} \left( \tfrac{1}{4} \right). 
\end{equation}
\end{Example}

\begin{Example}
The special case $a = 2m+1$ in (\ref{form-2}) gives
\begin{equation}
\int_{0}^{\pi/2} \sin^{2m}x \, dx = 2^{2m-1} B \left( m + \tfrac{1}{2}, m +
\tfrac{1}{2} \right), 
\end{equation}
\noindent
and using the identity
\begin{equation}
\Gamma \left( m + \tfrac{1}{2} \right) = \frac{\pi}{2^{2m}} \frac{(2m)!}{m!}
\end{equation}
\noindent
it yields
\begin{equation}
\int_{0}^{\pi/2} \sin^{2m}x \, dx = \frac{\binom{2m}{m} \, \pi}{2^{2m+1}}. 
\end{equation}
\noindent
This appears as $\mathbf{3.621.3}$. Similarly, $a= 2m+2$ in (\ref{form-2}) 
gives
\begin{equation}
\int_{0}^{\pi/2} \sin^{2m+1}x \, dx = 2^{2m}B(m+1,m+1),
\end{equation}
\noindent
that can be written as 
\begin{equation}
\int_{0}^{\pi/2} \sin^{2m+1}x \, dx = \frac{2^{2m}}{(2m+1)} \binom{2m}{m}^{-1}.
\end{equation}
\noindent
This is $\mathbf{3.621.4}$. 
\end{Example}

\begin{Example}
The integral $\mathbf{3.622.1}$ is 
\begin{eqnarray}
\int_{0}^{\pi/2} \tan^{\pm a}x \, dx & = & 
\int_{0}^{\pi/2} \sin^{\pm a}x \, \cos^{\mp a} x \, dx  \nonumber \\
 & = & \tfrac{1}{2} B \left( \tfrac{1 \pm a}{2}, 1 - \tfrac{1 \pm a}{2} \right) 
\nonumber \\
 & = & \tfrac{1}{2} \Gamma \left( \tfrac{1 \pm a}{2} \right) 
      \Gamma \left( 1- \tfrac{1 \pm a}{2} \right) 
\nonumber 
\end{eqnarray}
\noindent
and this reduces to 
\begin{equation}
\int_{0}^{\pi/2} \tan^{\pm a}x \, dx = \frac{\pi}{2 \, \cos( \pi a/2)},
\nonumber
\end{equation}
\noindent
as it appears in the table. 
\end{Example}

\begin{Example}
The identity
\begin{equation}
\tan^{a-1}x \, \cos^{2b-2}x = \sin^{a-1}x \, \cos^{2b-a-1}x 
\end{equation}
\noindent
shows that
\begin{equation}
\int_{0}^{\pi/2} \tan^{a-1}x \, \cos^{2b-2}x \, dx = 
\int_{0}^{\pi/2} \sin^{a-1}x \, \cos^{2b-a- 1}x \, dx = 
\frac{1}{2} B \left( \frac{a}{2}, b - \frac{a}{2} \right).
\end{equation}
\noindent
This appears as $\mathbf{3.623.1}$.
\end{Example}

\begin{Example}
The formula $\mathbf{3.624.2}$ states that
\begin{equation}
\int_{0}^{\pi/2} \frac{\sin^{a-1/2}x}{\cos^{2a-1}x} \, dx = 
\frac{\Gamma \left( \tfrac{a}{2} + \tfrac{1}{4} \right) \Gamma(1-a)}
{2 \Gamma \left( \tfrac{5}{4} - \tfrac{a}{2} \right) }. 
\end{equation}
\noindent
This comes directly from (\ref{36215}).
\end{Example}

\begin{Example}
The identity $\mathbf{3.627}$:
\begin{equation}
\int_{0}^{\pi/2} \frac{\tan^{a}x}{\cos^{a}x} \, dx = 
\int_{0}^{\pi/2} \frac{\cot^{a}x}{\sin^{a}x} \, dx = 
\frac{\Gamma(a) \Gamma( \tfrac{1}{2}-a)}{2^{a} \sqrt{\pi}} \sin \left( 
\frac{\pi a}{2} \right),
\label{3627}
\end{equation}
\noindent
can be verified by writing the first integral as
\begin{equation}
I = \int_{0}^{\pi/2} \sin^{a}x \, \cos^{1-2a}x \, dx = 
\frac{1}{2} B \left( \frac{a+1}{2}, \frac{1-2a}{2} \right). 
\end{equation}
\noindent
The beta function is 
\begin{equation}
\frac{1}{2} B \left( \frac{a+1}{2}, \frac{1-2a}{2} \right) = 
\frac{\Gamma \left( \tfrac{a}{2} + \tfrac{1}{2} \right) 
\Gamma \left( \tfrac{1}{2} - a \right) }
{2 \Gamma \left( 1 - \tfrac{a}{2} \right) }. 
\label{exp-1}
\end{equation}
\noindent
Using $\Gamma(t) \Gamma(1-t) = \frac{\pi}{\sin \pi t}$ we can reduce 
(\ref{exp-1}) to the expression in (\ref{3627}). 
\end{Example}

\begin{Example}
The evaluation of $\mathbf{3.628}$ 
\begin{equation}
\int_{0}^{\pi/2} \text{sec}^{2p}x \, \sin^{2p-1}x \, dx = 
\frac{\Gamma(p) \Gamma( \tfrac{1}{2} - p) }{ 2 \sqrt{\pi}},
\end{equation}
\noindent
is direct, once we write the integral as
\begin{equation}
\int_{0}^{\pi/2} \cos^{-2p}x \sin^{2p-1}x \, dx = 
\frac{1}{2} B \left( \tfrac{1}{2} - p , p \right). 
\end{equation}
\end{Example}

\section{A family of trigonometric integrals} \label{sec-trigo} 
\setcounter{equation}{0}

In this section we present the evaluation of a family of trigonometrical 
integrals in \cite{gr}. Many special cases appear in the table.

\begin{Prop}
Let $a, \, b, \, c \in \mathbb{R}$ with the condition 
\begin{equation}
a+b+2c+2 = 0. 
\label{res-1}
\end{equation}
\noindent
Then 
\begin{equation}
\int_{0}^{\pi/4} \sin^{a}x \, \cos^{b} x \, \cos^{c}(2x) \, dx = 
\frac{1}{2} B \left( \frac{a+1}{2}, c+1 \right).
\label{main-1}
\end{equation}
\end{Prop}
\begin{proof}
Let $t = \tan x$ to obtain 
\begin{equation}
\int_{0}^{\pi/4} \sin^{a}x \, \cos^{b} x \, \cos^{c}(2x) \, dx = 
\int_{0}^{1} t^{a} (1-t^{2})^{c} (1+t^{2})^{-(a+b+2c+2)/2} \, dt 
\end{equation}
\noindent
and (\ref{res-1}) yields
\begin{equation}
\int_{0}^{\pi/4} \sin^{a}x \, \cos^{b} x \, \cos^{c}(2x) \, dx = 
\int_{0}^{1} t^{a} (1-t^{2})^{c} \, dt.
\end{equation}
\noindent
The change of variables $s=t^{2}$ produces 
\begin{equation}
\int_{0}^{\pi/4} \sin^{a}x \, \cos^{b} x \, \cos^{c}(2x) \, dx = 
\frac{1}{2} \int_{0}^{1} s^{(a-1)/2} (1-s)^{c} \, ds,
\end{equation}
\noindent
and this last integral has the given beta value.
\end{proof}

\begin{Example}
The formula (\ref{main-1}), with $a= 2n, \, b=-2p-2n-2$ and $c=p$ appears as
$\mathbf{3.625.2}$ in \cite{gr}:
\begin{equation}
\int_{0}^{\pi/4} \frac{\sin^{2n}x \, \cos^{p}(2x)}{\cos^{2p+2n+2}x} \, dx = 
\tfrac{1}{2} B \left( n + \tfrac{1}{2}, p+1 \right).
\end{equation}
\end{Example}

\begin{Example}
The formula $\mathbf{3.624.3}$ 
\begin{equation}
\int_{0}^{\pi/4} \frac{\cos^{n-1/2}(2x)}{\cos^{2n+1}x} \, dx = \frac{\pi}
{2^{2n+1}} \binom{2n}{n}
\end{equation}
\noindent
corresponds to the case $a=0, \, b=-2n-1$ and $c=n- \tfrac{1}{2}$.
\end{Example}

\begin{Example}
Formula $\mathbf{3.624.4}$ in \cite{gr} 
\begin{equation}
\int_{0}^{\pi/4} \frac{\cos^{\mu}(2x)}{\cos^{2(\mu+1)}x} \, dx = 2^{2 \mu} 
B( \mu+1,\mu+1)
\label{36244}
\end{equation}
\noindent
corresponds to $a= 0, \, b=-2 \mu-2$ and $c = \mu$. Then (\ref{main-1}) 
gives
\begin{equation}
\int_{0}^{\pi/4} \frac{\cos^{\mu}(2x)}{\cos^{2(\mu+1)}x} \, dx = 
\frac{1}{2} B \left( \frac{1}{2}, \mu+1 \right).
\label{ans-1}
\end{equation}
\noindent
The duplication formula 
\begin{equation}
\Gamma(2x) = \frac{2^{2x-1}}{\sqrt{\pi}} \Gamma(x) \Gamma( x + \tfrac{1}{2} ),
\label{dupli-1}
\end{equation}
\noindent
transforms (\ref{ans-1}) into (\ref{36244}).
\end{Example}

\begin{Example}
The values $a= 2 \mu -2, \, b=0$ and $c=\mu$ produce $\mathbf{3.624.5}$:
\begin{equation}
\int_{0}^{\pi/4} \frac{\sin^{2 \mu -2}x}{\cos^{\mu}(2x)} \, dx = 
\frac{\Gamma( \mu - \tfrac{1}{2}) \Gamma(1- \mu)}{2 \sqrt{\pi}}
\end{equation}
\noindent
directly. Indeed, the answer from (\ref{main-1}) is  $B(\mu-1/2,1-\mu)/2$. The 
table also has the alternative answer $2^{1-2 \mu} B(2 \mu-1, 1-\mu)$ that 
can be obtained using (\ref{dupli-1}). 
\end{Example}

\begin{Example}
Formula $\mathbf{3.625.1}$:
\begin{equation}
\int_{0}^{\pi/4} \frac{\sin^{2n-1}x \, \cos^{p}(2x) }{\cos^{2p+2n+2}x} \, dx =
\frac{1}{2} B (n, p+1)
\end{equation}
\noindent
corresponds to $a=2n-1, \, b=-2p-2n-1$ and $c=p$.
\end{Example}

\begin{Example}
The choice $a=2n-1, \, b = -2n-2m$ and $c=m- \tfrac{1}{2}$ gives 
$\mathbf{3.625.3}$:
\begin{equation}
\int_{0}^{\pi/4} \frac{\sin^{2n-1}x \, \cos^{m-1/2}(2x) }{\cos^{2n+2m}x} \, 
dx = \frac{1}{2} B( n, m+ \tfrac{1}{2}). 
\end{equation}
\noindent
For $n, \, m \in \mathbb{N}$ we can also write
\begin{equation}
\int_{0}^{\pi/4} \frac{\sin^{2n-1}x \, \cos^{m-1/2}(2x) }{\cos^{2n+2m}x} \, 
dx = \frac{2^{2n-1}}{n} \binom{2m}{m} \binom{2n+2m}{n+m}^{-1} 
\binom{n+m}{n}^{-1}.
\end{equation}
\end{Example}

\begin{Example}
The values $a=2n, \, b = -2n-2m-1$ and $c = m - \tfrac{1}{2}$ give 
$\mathbf{3.625.4}$:
\begin{equation}
\int_{0}^{\pi/4} \frac{\sin^{2n}x \, \cos^{m - 1/2}(2x) }{\cos^{2n+2m+1}x} 
\, dx = \frac{1}{2} B \left( n + \tfrac{1}{2}, m + \tfrac{1}{2} \right). 
\end{equation}
\noindent
For $n, \, m \in \mathbb{N}$ we can also write
\begin{equation}
\int_{0}^{\pi/4} \frac{\sin^{2n}x \, \cos^{m - 1/2}(2x) }{\cos^{2n+2m+1}x} 
 \, dx = \frac{\pi}{2^{2n+2m+1}} \binom{2n}{n} \binom{2m}{m} 
\binom{n+m}{n}^{-1}.
\end{equation}
\end{Example}

\begin{Example}
Formula  $\mathbf{3.626.1}$:
\begin{equation}
\int_{0}^{\pi/4} \frac{\sin^{2n-1}x}{\cos^{2n+2}x} \sqrt{\cos(2x)} \, dx = 
\frac{1}{2} B( n, 3/2),
\end{equation}
\noindent
comes from (\ref{main-1}) with $a=2n-1, \, b = -2n-2$ and $c=1/2$. For 
$n \in \mathbb{N}$ we have
\begin{equation}
\int_{0}^{\pi/4} \frac{\sin^{2n-1}x}{\cos^{2n+2}x} \sqrt{\cos(2x)} \, dx = 
\frac{2^{2n} (n-1)! \, n!}{(2n+1)!}.
\end{equation}
\end{Example}

\begin{Example}
The last example in this section is formula  $\mathbf{3.626.2}$:
\begin{equation}
\int_{0}^{\pi/4} \frac{\sin^{2n}x}{\cos^{2n+3}x} \sqrt{\cos(2x)} \, dx = 
\frac{1}{2} B( n + \tfrac{1}{2}, \tfrac{3}{2}),
\end{equation}
\noindent
comes from (\ref{main-1}) with $a=2n, \, b = -2n-3$ and $c=1/2$. For 
$n \in \mathbb{N}$ we have
\begin{equation}
\int_{0}^{\pi/4} \frac{\sin^{2n}x}{\cos^{2n+3}x} \sqrt{\cos(2x)} \, dx = 
\frac{\pi}{2^{2n+2}} \frac{(2n)!}{n! \, (n+1)!}. 
\end{equation}
\end{Example}

\bibliography{../../../AllRef/biblio2}
\bibliographystyle{plain}
\end{document}